\documentclass[12pt]{amsart}

\usepackage{amssymb}
\usepackage{enumerate}

\makeatletter
\@namedef{subjclassname@2010}{%
  \textup{2010} Mathematics Subject Classification}
\makeatother

\newtheorem{thm}{Theorem}
\newtheorem{cor}[thm]{Corollary}
\newtheorem{lem}[thm]{Lemma}

\newtheorem*{thmJR}{Theorem JR}
\newtheorem*{thmS}{Theorem S}
\newtheorem*{xexa}{Example}

\newcommand{\length}{\operatorname{length}}

\begin{document}

\baselineskip=17pt


\title[Piecewise continuous mappings]{On piecewise continuous mappings of metrizable spaces}

\author[S. Medvedev]{Sergey Medvedev}

\address{Department of Mathematical and Functional Analysis \\ South Ural State University \\ pr. Lenina, 76,  Chelyabinsk, 454080 Russia}

\email{medvedevsv@susu.ru}

\date{}

\begin{abstract}
Let $f \colon X \rightarrow Y$ be a resolvable-measurable mapping of a metrizable space $X$ to a regular space $Y$. Then $f$ is piecewise continuous. Additionally, for a metrizable completely Baire space $X$, it is proved that $f$ is resolvable-measurable if and only if it is piecewise continuous.
\end{abstract}



\maketitle


In an old question Lusin asked if any Borel function is necessarily countably continuous. This question was answered negatively by Keldi\v{s} \cite{K34}, and an example of a Baire class 1 function which is not decomposable into countably many continuous functions was later found by Adyan and Novikov \cite{AN}; see also the paper of van Mill and Pol \cite{vMP}.

The first affirmative result was obtained by Jayne and Rogers \cite[Theorem 1]{JR}.

\begin{thmJR}[Jayne--Rogers]
If $X$ is an absolute Souslin-$\mathcal{F}$ set and $Y$ is a metric space, then $f \colon X \rightarrow Y$ is $\mathbf{\Delta}^0_2$-measurable if and only if it is piecewise continuous.
\end{thmJR}

Later Solecki \cite[Theorem 3.1]{S98} proved the first dichotomy theorem for Baire class 1 functions. This theorem shows how piecewise continuous functions can be found among $\mathbf{\Sigma}^0_2$-measurable ones.

\begin{thmS}[Solecki]
Let $f \colon X \rightarrow Y$ be a $\mathbf{\Sigma}^0_2$-measurable function from an analytic set  $X$ to a separable metric space $Y$. Then precisely one of the following holds:

\begin{enumerate}[\upshape (i)]
\item $f$ is piecewise continuous,

\item one of $L$, $L_1$ is contained in $f$, where $L$ and $L_1$ are two so-called Lebesgue's functions.
\end{enumerate}
\end{thmS}

Ka\v{c}ena, Motto Ros, and Semmes \cite[Theorem 1]{KMS} showed that Theorem JR holds for a regular space $Y$. They also got \cite[Theorem 8]{KMS} a strengthening of Solecki's theorem from an analytic set $X$ to an absolute Souslin-$\mathcal{F}$ set $X$.

On the other hand, Banakh and Bokalo \cite[Theorem 8.1]{BB} proved among other things that a mapping $f \colon X \rightarrow Y$ from a metrizable completely Baire space $X$ to a regular space $Y$ is piecewise continuous if and only if it is $\mathbf{\Pi}^0_2$-measurable. Under some set-theoretical assumptions, examples of $\mathbf{\Pi}^0_2$-measurable mappings which are not piecewise continuous were constructed in the work~\cite{BB}.

Recently, Ostrovsky \cite{Ost} proved that every resolvable-measurable function $f \colon X \rightarrow Y$ is countably continuous for any separable zero-dimensional metrizable spaces $X$ and $Y$.

The main result of the paper (see Theorem~\ref{t:4}) states that every resolvable-measurable mapping $f \colon X \rightarrow Y$ of a metrizable space $X$ to a regular space $Y$ is piecewise continuous. Comparison of our result and the Banakh and Bokalo theorem shows that the condition on $X$ is weakened but $f$ is restricted to the class of resolvable-measurable mappings. Notice also that Theorem~\ref{t:4} generalizes and strengthens the Ostrovsky theorem.

In completely metrizable spaces, resolvable sets coincide with $\mathbf{\Delta}^0_2$-sets, see \cite[p. 418]{Kur1}. Lemma~\ref{L:5} shows that every metrizable completely Baire space has such a property. This enables us to refine the above result of Banakh and Bokalo, see Theorem~\ref{t:6}.

Theorem~\ref{t:8} states that in the study of $\mathbf{\Sigma}^0_2$-measurable mappings defined on metrizable completely Baire spaces it suffices to consider separable spaces. In a sense, Theorem~\ref{t:8} is similar to the non-separable version of Solecki's Theorem~S.

\textbf{Notation}. For all undefined terms, see \cite{Eng}.

A subset $E$ of a space $X$ is \textit{resolvable} if it can be represented as
\begin{equation*}
E= (F_1 \setminus F_2) \cup (F_3 \setminus F_4) \cup \ldots \cup (F_\xi \setminus F_{\xi+1}) \cup  \ldots ,
\end{equation*}
where $\langle F_\xi \rangle$ forms a decreasing transfinite sequence of closed sets in $X$.

A metric space $X$ is said to be an \textit{absolute Souslin}-$\mathcal{F}$ set if $X$ is a result of the $\mathcal{A}$-operation applied to a system of closed subsets of $\widehat{X}$, where $\widehat{X}$ is the completion of $X$ under its metric. Metrizable continuous images of the space of  irrational numbers are called \textit{analytic sets}.

A mapping $f \colon X \rightarrow Y$ is said to be

\begin{enumerate}[$\bullet$]
\item \textit{resolvable-measurable} if $f^{-1}(U)$ is a resolvable subset of $X$ for every open set $U \subset Y$,

\item  $\mathbf{\Delta}^0_2$-\textit{measurable} if $f^{-1}(U) \in \mathbf{\Delta}^0_2(X)$ for every open set $U \subset Y$,

\item  $\mathbf{\Sigma}^0_2$-\textit{measurable} if $f^{-1}(U) \in \mathbf{\Sigma}^0_2(X)$ for every open set $U \subset Y$,

\item $\mathbf{\Pi}^0_2$-\textit{measurable} if $f^{-1}(U) \in \mathbf{\Pi}^0_2(X)$ for every open set $U \subset Y$,

\item  \textit{countably continuous} if $X$ can be covered by a sequence $X_0, X_1,  \ldots$ of sets such that the restriction $f \upharpoonright X_n$ is continuous for every $n \in \omega$,

\item  \textit{piecewise continuous} if $X$ can be covered by a sequence $X_0, X_1,  \ldots$ of closed sets such that the restriction $f \upharpoonright X_n$ is continuous for every $n \in \omega$.
\end{enumerate}

Obviously, every piecewise continuous mapping is countably continuous. Notice that every resolvable-measurable mapping of a metrizable space $X$ is  $\mathbf{\Sigma}^0_2$-measurable because, by \cite[p.~362]{Kur1}, every resolvable subset of a metrizable space $X$ is a $\mathbf{\Delta}^0_2$-set, i.e., a set that is both $F_\sigma$ and $G_\delta$ in $X$. The following example shows that there exists a $\mathbf{\Delta}^0_2$-measurable mapping which is not resolvable-measurable.

\begin{xexa}
Let $f \colon \mathbb{Q} \rightarrow D$ be a one-to-one mapping of the space $\mathbb{Q}$ of rational numbers onto the countable discrete space $D$.
Clearly, $f$ is piecewise continuous and  $\mathbf{\Delta}^0_2$-measurable. Gao and Kientenbeld \cite[Proposition 4]{GK} got a characterization of nonresolvable subsets of $\mathbb{Q}$. In particular, they showed that there exists a nonresolvable subset $A$ of $\mathbb{Q}$. Since $A = f^{-1}(f(A))$, the mapping $f$ is not resolvable-measurable.
\end{xexa}

The closure of a set $A \subset X$ is denoted by $\overline{A}$. Given a mapping $f \colon X \rightarrow Y$, let us denote by $\mathcal{I}_f$ the family of all subsets $A \subset X$ for which there is a set $S \in \mathbf{\Sigma}^0_2(X)$ such that $A \subset S$ and the restriction $f \upharpoonright S$ is piecewise continuous. In particular, $f$ is piecewise continuous if and only if $X \in \mathcal{I}_f$. From \cite[Proposition 3.5]{HZZ} it follows that the family $\mathcal{I}_f$ forms a $\sigma$-ideal which is $F_\sigma$ supported and is closed with respect to discrete unions, see also \cite{KMS}.

To prove Theorem~\ref{t:4}, we shall use the technique due to Ka\v{c}ena, Motto Ros, and Semmes \cite{KMS}. Therefore, the terminology from \cite{KMS} is applied. The sets $A, B \subset Y$ are \textit{strongly disjoint} if $\overline{A} \cap \overline{B} = \emptyset$. Let $f \colon X \rightarrow Y$ be a mapping. Put $A^f = f^{-1}(Y \setminus \overline{A})$. As noted in \cite{KMS}, if $A,B$ are strongly disjoint and $A^f , B^f \in \mathcal{I}_f$, then $X \in \mathcal{I}_f$.

Let $x \in X$, $X^\prime \subset X$, and $A \subset Y$. The pair $(x, X^\prime)$ is said to be $f$-\textit{irreducible outside} $A$ if for every open neighborhood $V \subset X$ of $x$ we have $A^f \cap X^\prime \cap V \notin \mathcal{I}_f$. Otherwise we say that $(x, X^\prime)$ is  $f$-\textit{reducible outside} $A$, i.e., there exist a neighborhood $V$ of $x$ and a set $S \in \mathbf{\Sigma}^0_2(X)$ such that $A^f \cap X^\prime \cap V \subset S$ and $f \upharpoonright S$ is piecewise continuous. Clearly, $x \in \overline{A^f \cap X^\prime}$ if $(x, X^\prime)$ is $f$-irreducible outside $A$.

\begin{lem}[{\cite[Lemma 3]{KMS}}]\label{lem1}
Let $X$ be a metrizable space and $Y$ a regular space. Suppose $f \colon X \rightarrow Y$ is a $\mathbf{\Sigma}^0_2$-measurable mapping, $X^\prime$ is a subset of $X$, and $A \subset Y$ is an open set such that $X^\prime \subset A^f$. Then the following assertions are equivalent:

\begin{enumerate}[\upshape(i)]

\item $X^\prime \notin \mathcal{I}_f$,

\item there exist a point $x \in \overline{X^\prime}$ and an open set $U \subset Y$ strongly disjoint from $A$ such that $f(x) \in U$ and the pair $(x, X^\prime)$ is $f$-irreducible outside $U$.
\end{enumerate}
\end{lem}

\begin{lem}[{\cite[Lemma 4]{KMS}}]\label{lem2}
Let $f \colon X \rightarrow Y$ be a mapping of a metrizable space $X$ to a regular space $Y$, $x \in X$, $X^\prime \subset X$, $A \subseteq Y$, and let $U_0, \ldots, U_k$ be a sequence of pairwise strongly disjoint open subsets of $Y$. If $(x; X^\prime)$ is $f$-irreducible outside $A$, then there is at most one $i \in \{ 0, \ldots, k \}$ such that $(x, X^\prime)$ is $f$-reducible outside $A \cup U_i$.
\end{lem}

Recall that a set $A \subset Y$ is \textit{relatively discrete} in $Y$ if for every point $a \in A$ there is an open set $U \subset Y$ such that $U \cap A = \{a \}$.

\begin{lem}\label{lem3}
Let $X$ be a metrizable space and $Y$ be a regular space. Suppose  $f \colon X \rightarrow Y$ is a $\mathbf{\Sigma}^0_2$-measurable mapping which is not piecewise continuous. Then there exists a subset $Z \subset X$ such that:
\begin{enumerate}[\upshape (1)]

\item $Z$ is homeomorphic to the space of rational numbers,

\item the restriction $f \upharpoonright Z$ is a bijection,

\item the set $f(Z)$ is relatively discrete in $Y$,

\item $\dim \overline{Z} = 0$.
\end{enumerate}
\end{lem}

\begin{proof}
Fix a metric $\rho$ on $X$. Denote by $2^{<\omega}$ the set of all binary sequences of finite length. The construction will be carried out by induction with respect to the order $\preceq$ on $2^{<\omega}$ defined by
\[s \preceq t \, \Longleftrightarrow \, \length(s) < \length(t) \vee (\length(s) = \length(t) \wedge s \leq_\mathrm{lex} t), \]
where $\leq_\mathrm{lex}$ is the usual lexicographical order on $2^{\length(s)}$. We write $s \prec t$ if $s \preceq t$ and $s \neq t$.

We will construct a sequence $\langle x_s \colon s \in 2^{<\omega} \rangle$ of points of $X$, a sequence $\langle V_s \colon s \in 2^{<\omega} \rangle$ of subsets of $X$, and a sequence $\langle U_s \colon s \in 2^{<\omega} \rangle$ of open subsets of $Y$ such that for every $s \in 2^{<\omega}$:

\begin{enumerate}[\upshape (1)]

\item if $t \subset s$ then $V_s \subset V_t$,

\item $V_s$ is an open ball in $X$ with the centre $ x_s$ and radius  $\leq 2^{- \length(s)}$,

\item if $s = t^\wedge 0$ then $x_s = x_t,$

\item $f(x_s) \in U_s$,

\item $(x_t, V_t)$ is $f$-irreducible outside $A$ for every $t \preceq s$, where $A = \bigcup_{u \preceq s}U_u$,

\item the family $\{V_t \colon t \in 2^{n} \}$ is pairwise strongly disjoint for every $n \in \omega$,

\item the family $\{U_t \colon t \preceq s \}$ is pairwise strongly disjoint.
\end{enumerate}

Since $f$ is not piecewise continuous, we can apply Lemma~\ref{lem1} with respect to $X^\prime = X$ and $A = \emptyset$ to obtain the point $x \in X$ and the open set $U \subset Y$. Then put $ x_\emptyset = x$ and $U_\emptyset = U$. Let $V_\emptyset= B(x_\emptyset, 1)$ be an open ball in $X$ with the centre $ x_\emptyset$ and radius 1.

Assume that $x_t$, $V_t$, and $U_t$ have been constructed for any $t \preceq s$. Put $x_{s^\wedge 0} = x_s$ and $U_{s^\wedge 0} = U_s$.

Let $A= \bigcup_{t \prec s^\wedge 1}U_t$ and $O= Y \setminus \overline{A}$.
By the inductive hypothesis, the pair $(x_s, V_s)$ is $f$-irreducible outside $A$. Take a neighborhood $W$ of $x_s$ such that $\overline{W} \subset V_s$. Then  $(x_s, W)$ is $f$-irreducible outside $A$ and $f^{-1}(O) \cap W = A^f \cap W \notin \mathcal{I}_f$. By Lemma~\ref{lem1} there exist a point $x^\prime \in \overline{f^{-1}(O) \cap W}$ and an open set $U_{x^\prime} \subset Y$ strongly disjoint from $A$ such that $f(x^\prime) \in U_{x^\prime}$ and the pair $(x^\prime, f^{-1}(O) \cap W)$ is $f$-irreducible outside $U_{x^\prime}$. Notice that $x^\prime \neq x_s$ because $f(x_s) \in A$ and $\overline{U_{x^\prime}} \cap \overline{A} = \emptyset$. If the pair $(x^\prime, f^{-1}(O) \cap W)$ is $f$-irreducible outside $A \cup U_{x^\prime}$, put $x^* = x^\prime$ and $U^* = U_{x^\prime}$.

Consider the case when the pair $(x^\prime, f^{-1}(O) \cap W)$ is $f$-reducible outside $A \cup U_{x^\prime}$. Take a neighborhood $W^\prime$ of $x^\prime$ such that $\overline{W^\prime} \subset V_s$. Let
\[O^\prime = Y \setminus (\overline{A} \cup \overline{U_{x^\prime}}) \text{  and } X^\prime = f^{-1}(O^\prime) \cap W \cap W^\prime.\]
Then the pair $(x^\prime, X^\prime)$ is $f$-irreducible outside $U_{x^\prime}$ and $X^\prime \notin \mathcal{I}_f$. As above, by Lemma~\ref{lem1} there exist a point $x^{\prime \prime} \in \overline{X^\prime}$ and an open set $U_{x^{\prime \prime}} \subset Y$ strongly disjoint from $A \cup U_{x^\prime}$ such that $f(x^{\prime \prime}) \in U_{x^{\prime \prime}}$ and the pair $(x^{\prime \prime}, X^\prime)$ is $f$-irreducible outside $U_{x^{\prime \prime}}$. Notice that $x^{\prime \prime} \neq x_s$ and $x^{\prime \prime} \neq x^\prime$. From Lemma~\ref{lem2} it follows that the pair $(x^{\prime \prime}, X^\prime)$ is $f$-irreducible outside $A \cup U_{x^{\prime \prime}}$. Then put $x^* = x^{\prime \prime}$ and $U^* = U_{x^{\prime \prime}}$.

Let $k = | \{t \in 2^{< \omega} \colon t \prec s^\wedge 1 \} |$, $z_0 = x^*$, and $U_0 = U^*$. Repeating the above construction, for $j = 0, \ldots, k$ recursively construct $z_j \in V_s$ and $U_j$ such that $f(z_j) \in U_j$, $U_j$ is strongly disjoint from $A_j= A \cup \bigcup_{i <j}U_i$, and the pair $(z_j, V_s \cap (A_j)^f)$ is $f$-irreducible outside $A \cup U_j$. From Lemma~\ref{lem2} it follows that for each $t \prec s^\wedge 1$ there is at most one $j \in \{ 0, \ldots, k \}$ such that $(x_t, V_t)$ is $f$-reducible outside $A \cup U_j$. The pigeonhole principle implies that there exists $\ell \in \{ 0, \ldots, k \}$ such that the pair $(z_\ell, V_s \cap (A_\ell )^f)$ is $f$-irreducible outside $A \cup U_\ell$ and  $(x_t, V_t)$ is $f$-irreducible outside $A \cup U_\ell$ for each $t \prec s^\wedge 1$. Finally, set $x_{s^\wedge 1} = z_\ell$ and $U_{s^\wedge 1} = U_\ell$.

Since $x_{s^\wedge 0}$ and $x_{s^\wedge 1}$ are two distinct points from $V_s$, we can choose their neighborhoods $V_{s^\wedge 0}$ and $V_{s^\wedge 1}$, respectively,  according to (1),(2), and (6).

One readily verifies that conditions (1)--(7) are satisfied.

The set $Z = \bigcup \{x_s \colon s \in 2^{<\omega} \}$ is countable and has no isolated points by (1) and (2). According to the Sierpi\'{n}ski theorem (see \cite[Exercise 6.2.A]{Eng}), $Z$ is homeomorphic to the space of rational numbers. By construction, the set $\bigcup \{f(x_s) \colon s \in 2^{<\omega}  \}$ consists of isolated points. From conditions (4) and (5) it follows that the restriction $f \upharpoonright Z$ is a bijection.

From conditions (1) and (2) it follows that the family $\mathcal{V}_n = \{V_t \colon t \in 2^n \}$ forms a cover of $Z$ by open sets of diameter $\leq 2^{1-n}$ for each $n \in \omega$. Then
$$\overline{Z} \subset \bigcap \bigl\{\bigcup \{\overline{V_t} \colon t \in 2^n \} \colon  n \in \omega \bigr\}.$$

Since the family $\mathcal{V}_n$ is finite and pairwise strongly discrete, we can find a pairwise strongly discrete open family $\mathcal{W}_n = \{W_t \colon t \in 2^{n} \}$ such that $\mathrm{diam}(W_t) < 2^{2-n}$ and  $\overline{V_t} \subset W_t$ for each $t \in 2^{n}$. Without loss of generality, each $\mathcal{W}_{n+1}$ is a refinement of $\mathcal{W}_n$. Every family $\{W \cap \overline{Z} \colon W \in \mathcal{W}_n \}$, $n \in \omega$, forms a discrete open cover of $\overline{Z}$. From the Vop\v{e}nka theorem (see \cite[Theorem 7.3.1]{Eng}) it follows that $\dim \overline{Z} = 0$.
\end{proof}

\begin{thm}\label{t:4}
Every resolvable-measurable mapping $f \colon X \rightarrow Y$ of a metrizable space $X$ to a regular space $Y$ is piecewise continuous.
\end{thm}

\begin{proof}
Suppose towards a contradiction that there is a resolvable-measur\-able mapping $f \colon X \rightarrow Y$ which is not piecewise continuous. Using Lemma \ref{lem3}, we can find a subset $Z \subset X$ such that $Z$ is homeomorphic to the space of rational numbers, the restriction $f \upharpoonright Z$ is a bijection, and $f(Z)$ is relatively discrete.
Since $f$ is a resolvable-measurable mapping, $f \upharpoonright Z$ is the same. On the other hand, $f \upharpoonright Z$ fails to be resolvable-measurable as shown in Example.
\end{proof}

\begin{cor}
Let $f \colon X \rightarrow Y$ be a bijection between metrizable spaces $X$ and $Y$ such that $f$ and $f^{-1}$ are both resolvable-measurable mappings. Then
$\dim X = \dim Y$.
\end{cor}

\begin{proof}
Theorem~\ref{t:4} implies that $X = \bigcup_{n \in \omega}A_n $, where each $A_n$ is closed in $X$  and each restriction $f \upharpoonright A_n$ is continuous. Similarly, $Y = \bigcup_{k \in \omega}B_k $, where each $B_k$ is closed in $Y$  and each restriction $f^{-1} \upharpoonright B_k$ is continuous. The sequence $\langle A_n \cap f^{-1}(B_k) \colon n \in \omega, k \in \omega \rangle$ forms a cover of $X$ by closed sets. Similarly, the sequence $\langle f(A_n) \cap B_k \colon n \in \omega, k \in \omega \rangle$ forms a cover of $Y$ by closed sets. Since $f \upharpoonright (A_n \cap f^{-1}(B_k))$ is a homeomorphism, we have
\[\dim (A_n \cap f^{-1}(B_k)) = \dim (f(A_n) \cap B_k).\]
The corollary follows from the countable sum theorem \cite[Theorem 7.2.1]{Eng}.
\end{proof}

A topological space $X$ is called a \textit{Baire space} if the intersection of countably many dense open sets in $X$ is dense; or equivalently every nonempty open set in $X$ is not of the first category. A space $X$ is \textit{completely Baire} if every closed subspace of $X$ is a Baire space. Recall that $F \subset X$ is a \textit{boundary set} in $X$ if its complement is dense, i.e., if $\overline{X \setminus F} = X$.

\begin{lem}\label{L:5}
For a metrizable space $X$ the following conditions are equivalent:

\begin{enumerate}[\upshape (i)]

\item no closed subspace of $X$ is homeomorphic to the space $\mathbb{Q}$ of rational numbers,

\item $X$ is a completely Baire space,

\item the $\mathbf{\Delta}^0_2(X)$-sets coincide with the resolvable sets in $X$.
\end{enumerate}

\end{lem}

\begin{proof}
(i)$\Rightarrow$(ii): Suppose towards a contradiction that $X$ is not a completely Baire space. Then there is a closed set $F \subset X$ which is not Baire. Hence we can find a nonempty open (in $F$) set $U \subset F$ of the first category in $F$. The closure $\overline{U}$ is of the first category on itself. According to \cite[Corollary 1]{M86} (see also \cite{D87}) $\overline{U}$ contains a closed copy of $\mathbb{Q}$, a contradiction.

(ii)$\Rightarrow$(iii): By \cite[p.~362]{Kur1}, every resolvable set in a metrizable space is a $\mathbf{\Delta}^0_2$-set.

Conversely, let $E \in \mathbf{\Delta}^0_2(X)$ and $F$ be an arbitrary non-empty closed set. According to \cite[p.~99]{Kur1}, we have to show that that either $F \cap E$ or $F \setminus E$ is not a boundary set in $F$. Otherwise, the sets $F \cap E$ and $F \setminus E$ would be of the first category in $F$ (because every boundary $\mathcal{F}_\sigma$-set is of the first category), so their union $F = (F \cap E) \cup (F \setminus E)$ would be of the first category on $F$. This contradicts the fact that $F$ is a Baire space.

(iii)$\Rightarrow$(i): Striving for a contradiction, suppose that $X$ contains a closed set $F$ which is homeomorphic to $\mathbb{Q}$. As shown in Example, there is a nonresolvable set $A \in \mathbf{\Delta}^0_2(F)$. The set $A$ is the same in $X$ because $F$ is closed in $X$.
\end{proof}

\begin{thm}\label{t:6}
Let $f \colon X \rightarrow Y$ be a mapping of a metrizable completely Baire space $X$ to a regular space $Y$. Then the following conditions are equivalent:

\begin{enumerate}[\upshape (i)]

\item $f$ is resolvable-measurable,

\item $f$ is piecewise continuous,

\item $f$ is $\mathbf{\Pi}^0_2$-measurable.
\end{enumerate}
\end{thm}

\begin{proof}
The implication (i)$\Rightarrow$(ii) follows from Theorem~\ref{t:4}.

(ii)$\Rightarrow$(i): By definition, there are closed sets $X_n \subset X$, $n \in \omega$, such that $\bigcup_{n \in \omega}X_n = X$ and each $f \upharpoonright X_n$ is continuous. Then
$$f^{-1}(A) = \bigcup \{X_n \cap f^{-1}(A) \colon n \in \omega \}$$
is an $\mathcal{F}_\sigma$-set in $X$ for every open (or closed) set $A \subset Y$. Hence $f^{-1}(U) \in \mathbf{\Delta}^0_2(X)$ for every open $U \subset Y$. From Lemma~\ref{L:5} it follows that $f^{-1}(U)$ is a resolvable set in $X$.

Banakh and Bokalo \cite[Theorem 8.1]{BB} got (ii)$\: \Leftrightarrow$ (iii).
\end{proof}

\begin{cor}
Let $X$ be a completely metrizable space and $Y$ a regular space.
Then $f \colon X \rightarrow Y$ is resolvable-measurable if and only if $f$ is $\mathbf{\Pi}^0_2$-measurable.
\end{cor}

According to \cite[Corollary 6]{KMS}, for an absolute Souslin-$\mathcal{F}$ set $X$, if $f \colon X \rightarrow Y$ is $\mathbf{\Sigma}^0_2$-measurable and not piecewise continuous, then there is a copy $K \subset X$ of the Cantor space $2^\omega$ such that $f \upharpoonright K$ has the same properties.
The following theorem shows that a similar statement is valid for metrizable completely Baire spaces. However, such a set $K$ from Theorem~\ref{t:8} need not be homeomorphic to the Cantor space. In fact, every Bernstein set is a metrizable completely Baire space but it contains no copy of the Cantor space.

\begin{thm}\label{t:8}
Let $X$ be a metrizable completely Baire space and $Y$ a regular space. If
$f \colon X \rightarrow Y$ is $\mathbf{\Sigma}^0_2$-measurable and not piecewise continuous, then there is a zero-dimensional separable closed set $K \subset X$ such that the restriction $f \upharpoonright K$ is the same.
\end{thm}

\begin{proof}
Let $K = \overline{Z}$, where the set $Z \subset X$ is obtained by Lemma \ref{lem3}. Clearly, $f \upharpoonright K$ is $\mathbf{\Sigma}^0_2$-measurable.

Suppose towards a contradiction that $f \upharpoonright K$ is piecewise continuous. Then there are closed sets $K_n \subset X$, $n \in \omega$, such that $\bigcup_{n \in \omega}K_n = K$ and $f \upharpoonright K_n$ is continuous. Since $K$ is a Baire space, there exists a $K_j$ with the nonempty interior $V_j$ (in $K$). Clearly, $f \upharpoonright \overline{V_j \cap Z}$ is continuous.  Take a point $q \in V_j \cap Z$. Fix a neighborhood $U_q \subset Y$ of $f(q)$ such that $U_q \cap f(Z) = f(q)$. From continuity of $f \upharpoonright \overline{V_j \cap Z}$ it follows that there is a neighborhood $V \subset V_j$ (in $K$) of $q$ such that $f(V) \subset U_q$. Then $V \cap Z = \{ q \}$, i.e., $q$ is an isolated point of $Z$. This contradicts the fact that the set $V_j \cap Z$ has no isolated points.
\end{proof}

The last theorem yields

\begin{thm}
Let $f \colon X \rightarrow Y$ be an $F_\sigma$-measurable mapping of a metrizable completely Baire space $X$ to a regular space $Y$. If the restriction $f \upharpoonright Z$ is piecewise continuous for any zero-dimensional separable closed subset $Z$ of $X$, then $f$ is piecewise continuous.
\end{thm}

\end{document}